\documentclass[a4paper]{amsart}

\usepackage[utf8]{inputenc}
\usepackage[T1]{fontenc}
\usepackage{amsfonts}
\usepackage{amsmath}
\usepackage{amssymb}
\usepackage{amstext}
\usepackage{amsthm}
\usepackage[shortlabels]{enumitem}
\usepackage{geometry}
\usepackage{fullpage}
\usepackage{hyperref}
\usepackage{cleveref}
\usepackage{mathtools}
\usepackage{tikz}
\usetikzlibrary{matrix,arrows,decorations.pathmorphing}
\usepackage{pgf}
\usepackage{xcolor}
\usepackage{shuffle}
\usepackage{ifthen}
\usepackage{footmisc}
\usepackage{pifont} 
\usepackage{caption} 

\usepackage{algorithm}
\usepackage{algpseudocode}

\algrenewcommand\algorithmicrequire{\textbf{Input:}}
\algrenewcommand\algorithmicensure{\textbf{Output:}}

\newtheorem{theorem}{Theorem}[section]
\newtheorem{lemma}[theorem]{Lemma}

\newtheorem{corollary}[theorem]{Corollary}
\newtheorem{proposition}[theorem]{Proposition}

\theoremstyle{definition}
\newtheorem{definition}[theorem]{Definition}
\newtheorem{example}[theorem]{Example}

\theoremstyle{remark}
\newtheorem{remark}[theorem]{Remark}

\numberwithin{equation}{section}

\title[Shuffle theorem and sandpiles]{On a bijection of Loehr and Remmel}

\author{Michele D'Adderio}
\address{Universit\`a di Pisa\\Dipartimento di Matematica\\ Largo Bruno Pontecorvo 5, 56127 Pisa\\ Italy}\email{michele.dadderio@unipi.it}

\author{Alessio Sgubin}
\address{LIGM, Universit\'e Gustave-Eiffel\\ CNRS, ENPC, ESIEE-Paris\\ 5 Boulevard Descartes,
	Champs-sur-Marne, 77454 Marne-la-Vall\'ee cedex 2\\ FRANCE}\email{a.sgubin@studenti.unipi.it }

\newcommand{\area}{\mathsf{area}}
\newcommand{\bounce}{\mathsf{bounce}}
\newcommand{\dinv}{\mathsf{dinv}}
\newcommand{\pmaj}{\mathsf{pmaj}}

\newcommand{\maj}{\mathsf{maj}}

\newcommand{\PF}{\mathsf{PF}}

\DeclareMathOperator{\Dyck}{Dyck}
\DeclareMathOperator{\Dinv}{Dinv}
\DeclareMathOperator{\rev}{rev}

\def\Sy{\mathfrak{S}}
\def\ldread{\ensuremath{<_{\text{d-read}}}}
\def\lread{\ensuremath{<_{\text{d-read}}}}
\def\wdread{\ensuremath{w_{\text{d-read}}}}
\def\wpread{\ensuremath{w_{\text{p-read}}}}
\def\PFp{\ensuremath{\overline{\text{PF}}}}

\newcommand{\defeq}{\vcentcolon=}

\newcommand{\dyckpath}[3]{
	\pgfmathsetmacro{\unit}{1}
	\pgfmathsetmacro{\n}{#1}
	\pgfmathsetmacro{\k}{#2}
	\pgfmathsetmacro{\nk}{#1 * #2}
	\foreach \row in {1,...,#1}{
		\foreach \h in {1,...,#2}{
			\fill[yellow, nearly transparent] (\row * \k * \unit + \h*\unit - \k*\unit - \unit, \row*\unit - \unit) rectangle (\row * \k *\unit + \h*\unit - \k*\unit, \row*\unit);
		}
	}
	\foreach \xcoord in {0,...,\nk}{
		\draw[gray] (\xcoord*\unit,0) -- (\xcoord*\unit,\n*\unit);
	}
	\foreach \ycoord in {0,...,#1}{
		\draw[gray] (0,\ycoord*\unit) -- (\nk*\unit,\ycoord*\unit);
	}
	\foreach \col [count=\row] in {#3}{
		\draw[red, line width = 1.6pt] (\nk*\unit - \col*\unit,\row*\unit - \unit) -- (\nk*\unit - \col*\unit,\row*\unit);
	}
	\coordinate (R0C0) at (0,0);
	\foreach \col [count=\row] in {#3}{
		\draw[red,line width = 1.6pt] (\nk*\unit - \col*\unit,\row*\unit - \unit) -- (\nk*\unit - \col*\unit,\row*\unit);
		
		\coordinate (R\row C\row) at (\nk*\unit - \col*\unit,\row*\unit);
		\pgfmathtruncatemacro{\minus}{\row - 1};
		\draw[red,line width = 1.6pt] (R\minus C\minus) -- (\nk*\unit - \col*\unit,\row*\unit - \unit);
	}
	\draw[red,line width = 1.6pt] (R\n C\n) -- (\nk*\unit,\n*\unit);
}

\newcommand{\parkfunc}[5]{
	\pgfmathsetmacro{\unit}{#5}
	\pgfmathsetmacro{\sizefont}{#5 / 0.7}
	\pgfmathsetmacro{\n}{#1}
	\pgfmathsetmacro{\k}{#2}
	\pgfmathsetmacro{\nk}{#1 * #2}
	\foreach \row in {1,...,#1}{
		\foreach \h in {1,...,#2}{
			\fill[yellow, nearly transparent] (\row * \k * \unit + \h*\unit - \k*\unit - \unit, \row*\unit - \unit) rectangle (\row * \k *\unit + \h*\unit - \k*\unit, \row*\unit);
		}
	}
	\foreach \xcoord in {0,...,\nk}{
		\draw[gray] (\xcoord*\unit,0) -- (\xcoord*\unit,\n*\unit);
	}
	\foreach \ycoord in {0,...,#1}{
		\draw[gray] (0,\ycoord*\unit) -- (\nk*\unit,\ycoord*\unit);
	}
	\foreach \col [count=\row] in {#3}{
		\draw[red,very thick] (\nk*\unit - \col*\unit,\row*\unit - \unit) -- (\nk*\unit - \col*\unit,\row*\unit);
	}
	\coordinate (R0C0) at (0,0);
	\foreach \col [count=\row] in {#3}{
		\draw[red,very thick] (\nk*\unit - \col*\unit,\row*\unit - \unit) -- (\nk*\unit - \col*\unit,\row*\unit);
		\coordinate (R\row C\row) at (\nk*\unit - \col*\unit,\row*\unit);
		\coordinate (L\row) at (\nk*\unit - \col*\unit + .5*\unit,\row*\unit - .5*\unit);
		\pgfmathtruncatemacro{\minus}{\row - 1}
		\draw[red,very thick] (R\minus C\minus) -- (\nk*\unit - \col*\unit,\row*\unit - \unit);
	}
	\draw[red,very thick] (R\n C\n) -- (\nk*\unit,\n*\unit);
	\foreach \label [count=\row] in {#4}{
		\node[red,thick,font=\bfseries,scale={\sizefont}] at (L\row) {\label};
	}
}

\begin{document}
	
\maketitle
	
{\centering \emph{Dedicated to the memory of Adriano Garsia}.\par}

\begin{abstract}
In this expository article we review a remarkable bijection $\phi_n$ due to Loehr and Remmel from the set of parking functions of size $n$ into itself, which sends the bistatistic $(\dinv,\area)$ into the bistatistic $(\area,\pmaj)$. The only novelty of the present work is our definition of $\phi_n$, which is more direct than the original one, hence easier to compute and to work with. 
\end{abstract}

\section{Introduction}

In a recent breakthrough~\cite{CarlssonMellitShuffle} Carlsson and Mellit gave a positive solution to the long-standing \emph{shuffle conjecture}~\cite{HHLRU-2005}, which gives a combinatorial formula for the Frobenius characteristic of the so-called diagonal harmonics. More precisely, this theorem provides the monomial expansion of the symmetric function $\nabla e_n$, where $e_n$ is the elementary symmetric function of degree $n$ in the variables $x_1,x_2,\dots$, and $\nabla$ is the famous \emph{nabla} operator introduced by Bergeron and Garsia in the 90's. In this formula, to each \emph{labelled Dyck path} in the $n\times n$ grid corresponds a monomial, where the variables $x_1,x_2,\dots$ keep track of the labels, while the variables $q$ and $t$ keep track of the bistatistic ($\dinv$, $\area$).
\smallskip

The statistic $\dinv$, discovered by Haiman, was originally defined only on ``unlabelled'' Dyck paths, and together with $\area$ provided a combinatorial interpretation of the famous $(q,t)$-Catalan $\langle \nabla e_n,e_n\rangle$ (here $\langle - ,-\rangle$ denotes the Hall scalar product on symmetric functions). At the same time Hanglund defined a statistic $\bounce$ on Dyck paths, that together with $\area$ provided another interpretation of the $(q,t)$-Catalan. These two interpretations can be proved to be the same by an explicit bijection, usually called $\zeta$, due to Haglund and Loehr~\cite{Haglund_Conjectured,Haglund_Loehr_Conjectured}, sending a Dyck path in the $n\times n$ grid into another one whose bistatistic $(\area,\bounce)$ coincides with the bistatistic $(\dinv,\area)$ of the original one. 

We notice here that both the statistic $\bounce$ and the bijection $\zeta$ have been extended first by Loehr~\cite{Loehr_HigherCatalan} to Dyck paths in the $n\times kn$ grid, and then to more general Dyck paths in a rectangular grid, the maps going under the collective name of \emph{sweep maps}, see e.g.~\cite{ArmstrongLoehrWarrington_Sweep,ThomasWilliams_Sweeping}.
\smallskip

In~\cite{LoehrRemmel_pmaj} Loehr and Remmel introduced a new statistic $\pmaj$ for labelled Dyck paths in the $n\times n$ grid, as an extension of the $\bounce$ to the labelled objects. In this way they provided an alternative combinatorial interpretation of $\nabla e_n$: their formula is in terms of the same objects, but using the bistatistic ($\area$, $\pmaj$). In fact, they proved that the two combinatorial formulas coincide, by defining a bijection sending the bistatistic $(\dinv,\area)$ to $(\area, \pmaj)$. This bijection is a far reaching extension  to labelled Dyck paths of the original $\zeta$ map. We like to stress that both the pmaj statistic and the bijection introduced by Loehr and Remmel are truly remarkable! 
\medskip

The goal of this expository article is to present this beautiful bijection, and prove its fundamental properties. The motivation to write it comes from the only original contribution that it contains: indeed we provide a new definition of the map, that we denote $\phi_n$, which is more direct than the original one, hence easier to compute, to implement, and more generally to work with.

Before diving into the bulk of the present article, we would like to make a few comments on both the pmaj statistic and our definition of the bijection $\phi_n$.
\medskip

A distinctive feature of the pmaj statistic is its ``algorithmic'' definition, in stark contrast to the more ``static'' area and dinv. In our opinion, the reinterpretation given in~\cite{DDILLV} of this statistic as a \emph{delay} statistic defined on sorted recurrent configuration of the sandpile model (a ``dynamical'' object) on complete graphs sheds light on the mysterious nature of this remarkable statistic.

In a forthcoming article~\cite{DAdderioSgubinNablakSandpile}, we will define a pmaj statistic on labelled Dyck paths in the $n\times kn$ grid with labels $[n]=\{1,2,\dots,n\}$, and we will provide a bijection $\phi_n^{(k)}$ of this set into itself, which extends $\phi_n$, i.e.\ $\phi_n=\phi_n^{(1)}$, and which has the same properties of $\phi_n=\phi_n^{(1)}$, i.e.\ it sends $(\dinv,\area)$ (cf.~\cite{Mellit_Rational}) into $(\area,\pmaj)$. This will provide a combinatorial formula for $\nabla^k e_n$ in terms of $(\area,\pmaj)$, which is new for $k\geq 2$. Moreover, in~\cite{DAdderioSgubinNablakSandpile}, we will provide also a sandpile interpretation, extending the main results in~\cite{DDILLV}. It should be noticed that the sandpile interpretation has been instrumental to find the pmaj in this more general setting, while our new definition of $\phi_n$ has been instrumental to find the extension $\phi_n^{(k)}$ for $k\geq 2$.

\subsection*{Acknowledgments}

D'Adderio is partially supported by PRIN 2022A7L229 ALTOP, by INDAM research group GNSAGA, and by ARC “From algebra to combinatorics, and back”. 

Sgubin is Co-Funded by the European Union. Views and opinions expressed are however those of the author(s) only and do not necessarily reflect those of the European Union. Neither the European Union nor the granting authority can be held responsible for them.

\section{Parking functions and their statistics}\label{sec:cl_definitions}

Throughout this article, $n \in \mathbb{N} \setminus\{0\}$ will be a positive natural number and we set $[n] \defeq \{1,\dots,n\}$. We start by defining parking functions and the three statistics involved in the Loehr-Remmel bijection. 

\begin{definition}\label{def:cl_parkfunc}
	A \emph{parking function} of size $n$ is a function $f: [n] \to [n]$ such that $|f^{-1}(\{1,\dots,m\})| \geq m$ for every index $m \in [n]$. We denote by $\PF_n$ the set of all parking functions of size $n$.
\end{definition}

\begin{example}\label{exa:cl_pf}
	Consider the function $f: [6] \to [6]$ defined by:
	\[
	\begin{array}{lllll}
		f(1) = 3 & \quad & f(2) = 3 & \quad & f(3) = 6\\
		f(4) = 1 & & f(5) = 1 & & f(6) = 3.
	\end{array}
	\]
	One can directly compute the pre-images and check the parking function condition
	\[
	\begin{array}{lllll}
		f^{-1}([1]) = \{4,5\} & \quad & f^{-1}([2]) = \{4,5\} & \quad & f^{-1}([3]) = \{1,2,4,5,6\}\\
		f^{-1}([4]) = \{1,2,4,5,6\} & \quad & f^{-1}([5]) = \{1,2,4,5,6\} & \quad & f^{-1}([6]) = \{1,2,3,4,5,6\},
	\end{array}
	\]
	thus $f \in \PF_6$.
\end{example}

Parking functions have a useful pictorial representation as labelled Dyck paths. 

\begin{definition}\label{def:cl_Dyckpath}
	A \emph{$n$-Dyck path} is a lattice path consisting of \emph{north steps} (going from a lattice point $(a,b)$ to the point $(a,b+1)$) and \emph{east steps} (going from a lattice point $(a,b)$ to $(a+1,b)$) going from $(0,0)$ to $(n,n)$ that never goes below the \emph{main diagonal} $y=x$. We denote by $\Dyck_n$ the set of all $n$-Dyck paths.\\
	The \emph{diagram} of a $n$-Dyck path is its picture in the $n\times n$ square grid that bounds it.
\end{definition}

The diagram of a $6$-Dyck path is shown in Figure~\ref{fig:cl_Dyckpaths}. The squares intersecting the main diagonal are shaded in yellow.
	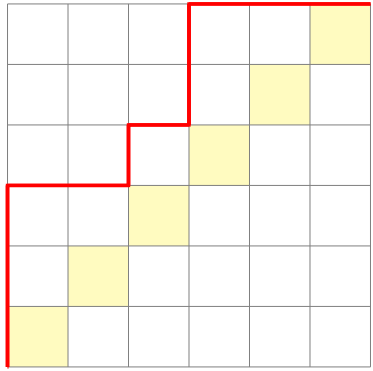
\begin{figure}[ht]
		\centering
		\begin{tikzpicture}
			\node[scale=.8] at (0,0) {
				\begin{tikzpicture}
					\dyckpath{6}{1}{6,6,6,4,3,3}
					\draw[red,line width = 1.6pt] (0,0)|-(2,3)|-(3,4)|-(6,6);
				\end{tikzpicture}
			};
		\end{tikzpicture}
		\caption{An example of $6$-Dyck paths.}\label{fig:cl_Dyckpaths}
	\end{figure}

Given a $n$-Dyck path $D$, we will refer to the horizontal and vertical sequences of $n$ unit squares in its diagram as the \emph{rows} and the \emph{columns} of $D$, respectively.

\begin{definition}\label{def:cl_labelledDyck}
	A \emph{labelled} $n$-Dyck path is a Dyck path $D \in \Dyck_n$ where each north step is labelled by a positive integer, so that the labels of north steps on the same vertical line have increasing labels when read from bottom to top. The \emph{diagram} of a labelled $n$-Dyck path is the diagram of the associated Dyck path with the label of each vertical step appearing in the unit square to its right.
\end{definition}
\noindent The diagram of a labelled $6$-Dyck path is shown in Figure~\ref{fig:cl_exa_labDyck}.
\begin{figure}[htbp]
	\centering
	\begin{tikzpicture}
		\node[scale=.8] at (0,0) {
			\begin{tikzpicture}
				\parkfunc{6}{1}{6,6,6,4,3,3}{6,7,11,2,3,8}{1};
				\draw[red,line width = 1.6pt] (0,0)|-(2,3)|-(3,4)|-(6,6);
		\end{tikzpicture}};
	\end{tikzpicture}
	\caption{An example of labelled $6$-Dyck path.}\label{fig:cl_exa_labDyck}
\end{figure}
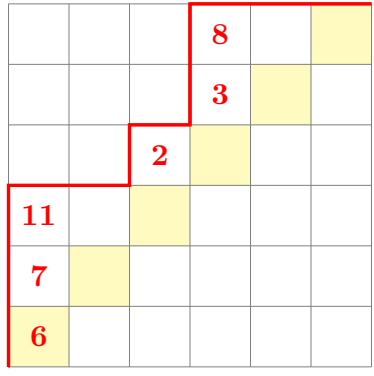
\smallskip

Finally, we can give the pictorial representation of parking functions.

\begin{remark}\label{rem:cl_pf_Dyck}
There is a natural bijection between parking functions of size $n$ and labelled $n$-Dyck paths with set of labels $[n]$: to a parking function $f \in \PF_n$ we associate the (unique) labelled $n$-Dyck path for which $f^{-1}(i)$ is the set of labels of its north steps on the line $x=i-1$ for all $i\in [n]$. Its inverse is also easy: if we number from $1$ to $n$ (from left to right) the columns of such a labelled $n$-Dyck path, then we associate to it the function sending each label in the number of its column.\smallskip\\
	We show an example of this correspondence in Figure~\ref{fig:cl_exa_PFDyck}.
	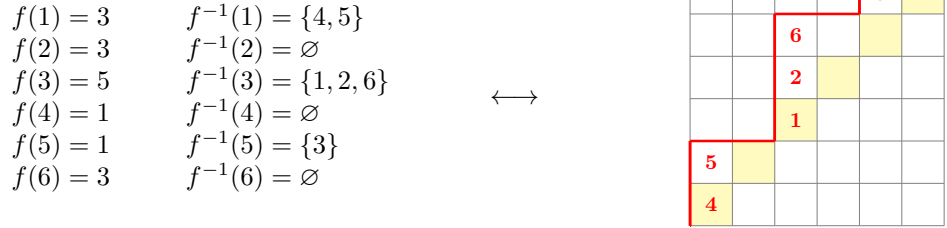
\begin{figure}[ht]
		\centering
		\begin{tikzpicture}
			\node at (0,0) {
				$\begin{array}{c}
					f(1) = 3\\
					f(2) = 3\\
					f(3) = 5\\
					f(4) = 1\\
					f(5) = 1\\
					f(6) = 3
				\end{array}$
			};
			\node at (3,0) {
				$\begin{array}{l}
					f^{-1}(1) = \{4,5\}\\
					f^{-1}(2) = \varnothing\\
					f^{-1}(3) = \{1,2,6\}\\
					f^{-1}(4) = \varnothing\\
					f^{-1}(5) = \{3\}\\
					f^{-1}(6) = \varnothing
				\end{array}$
			};
			\node at (6,0) {$\longleftrightarrow$};
			\node[scale=.8] at (10,0) {
				\begin{tikzpicture}
					\parkfunc{6}{1}{6,6,4,4,4,2}{4,5,1,2,6,3}{.7}
				\end{tikzpicture}
			};
		\end{tikzpicture}
		\caption{Example of the correspondence: from the parking function $f \in \PF_6$ to the labelled path.}
		\label{fig:cl_exa_PFDyck}
	\end{figure}
\end{remark}

\medskip

\underline{Warning}: For the rest of this article, we will identify parking functions of size $n$ and the labelled $n$-Dyck paths corresponding to them under the above bijection, referring to either of them as ``parking functions''.

\medskip

Now we can introduce the statistics relevant to the Loehr-Remmel bijection. The first one is classic.

\begin{definition}\label{def:cl_area}
Given a Dyck path $D \in \Dyck_n$, its \emph{area word} is the $n$-uple $w_{\area}(D) \defeq (w_1,\dots,w_n)$ where for each $i \in [n]$:
	\[
	w_i \defeq \#\text{ of integer squares in the $i^{\text{th}}$-row between the path and the main diagonal},
	\]
	where we number the rows of $D$ from $1$ to $n$, from bottom to top. We define the \emph{area} of $D \in \Dyck_n$ as
	\[
	\area(D) \defeq w_1 + \dots + w_n.
	\]
	Hence the area of a Dyck path is simply the number of squares between the path and the main diagonal.\\
	Given a parking function $f \in \PF_n$, we denote by $D(f)$ the Dyck path obtained by ignoring the labels of its vertical steps. We define the \emph{area word} and the \emph{area} of $f$, respectively, as
	\[
	w_{\area}(f) \defeq w_{\area}(D(f)) \qquad \qquad \area(f) \defeq \area(D(f)).
	\]
	Since rows are associated to the labels of a parking function, if the $i^{\text{th}}$-row has a north step labelled $\lambda_i \in [n]$ then we denote the \emph{row-area contribution of $\lambda_i$} by $a_{\lambda_i}(f) \defeq w_i$.
\end{definition}

\begin{example}\label{exa:cl_area}
	Consider the parking function $f \in \PF_6$ appearing in Figure~\ref{fig:cl_exa_PFDyck}. We compute its area word and $\area(f)$ in Figure~\ref{fig:cl_exa_area}.
	\begin{figure}[hbtp]
		\centering
		\begin{tikzpicture}
			\parkfunc{6}{1}{6,6,4,4,4,2}{4,5,1,2,6,3}{.7}
			\node[right] at (4.7, .35) {$w_1 = 0 = a_4(f)$};
			\node[right] at (4.7,1.05) {$w_2 = 1 = a_5(f)$};
			\node[right] at (4.7,1.75) {$w_3 = 0 = a_1(f)$};
			\node[right] at (4.7,2.45) {$w_4 = 1 = a_2(f)$};
			\node[right] at (4.7,3.15) {$w_5 = 2 = a_6(f)$};
			\node[right] at (4.7,3.85) {$w_6 = 1 = a_3(f)$};
			\node at (11,3) {$
				\begin{array}{rcl}
					w_{\area}(f) & = &  (0,1,0,1,2,1)\\
					 & \ & \\
					\area(f) & = & 0 + 1 + 0 + 1 + 2 + 1\\
					& = & 5
				\end{array}
				$};
		\end{tikzpicture}
		\caption{Computation of the area word and $\area$ statistic of a parking function.}\label{fig:cl_exa_area}
	\end{figure}
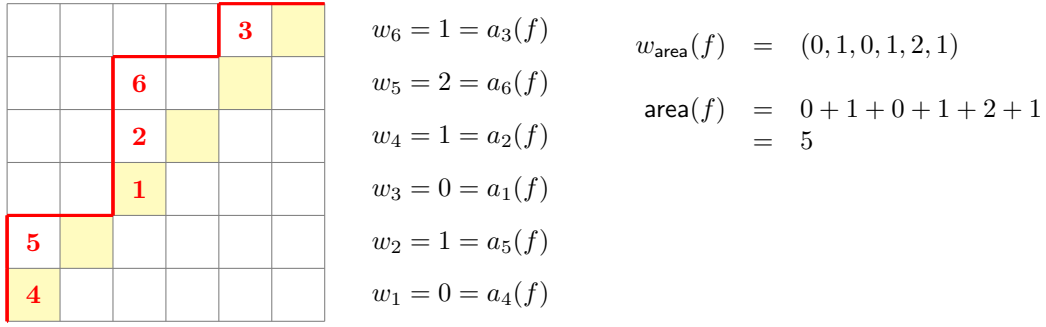
\end{example}

While the area statistic of a parking function $f\in \PF_n$ depends only on the corresponding Dyck path $D(f)$, the next two statistics that we introduce are going to depend on the labels of $f$ as well.
\medskip

The dinv statistic was introduced by Haiman.

\begin{definition}\label{def:cl_dinv}
Given a parking function $f \in \PF_n$, we define two conditions for unordered pairs of distinct labels $\{\lambda,\mu\} \subseteq [n]$ with $\lambda < \mu$:
	\begin{enumerate}[label=(\Alph*)]
		\item the labels satisfy $f(\lambda) < f(\mu)$ and $a_\lambda(f) = a_\mu(f)$,
		\item the labels satisfy $f(\lambda) > f(\mu)$ and $a_\lambda(f) + 1 = a_\mu(f)$.
	\end{enumerate}
	Then the \emph{diagonal inversion set} of $f$ is defined as
\[\Dinv(f) \defeq \big\{\{\lambda,\mu\} \subseteq [n]\text{ with  }\lambda<\mu \text{ that satisfy (A) or (B)}\big\},\]
while the \emph{dinv} of $f$ is defined as
\[\dinv(f) \defeq \# \Dinv(f).\]
\end{definition}

\begin{example}\label{exa:cl_dinv}
	Consider the parking function $f \in \PF_6$ from Figure~\ref{fig:cl_exa_area}. The set of diagonal inversions in this case is
	\[
	\Dinv(f) = \{ {\color{green!60!black} \{1,3\}}, {\color{blue} \{1,5\}}, {\color{blue} \{3,5\}} \}
	\]
	where we denote in green the diagonal inversions of type {\color{green!60!black} (A)} and in blue the ones of type {\color{blue} (B)}. Hence $\dinv(f)=3$.\\
	In Figure~\ref{fig:cl_exa_dinv} we give a representation of the types of diagonal inversions, using this specific example.
	\begin{figure}[ht]
		\centering
		\begin{tikzpicture}
			\node at (0,0) {\begin{tikzpicture}
					\parkfunc{6}{1}{6,6,4,4,4,1}{4,5,1,2,6,3}{.7}
					\draw[green!60!black,very thick] (1.75,1.75) circle (.3);
					\draw[green!60!black,very thick] (3.85,3.85) circle (.3);
					\draw[green!60!black,very thick] (2,2) -- (3.6,3.6);
			\end{tikzpicture}};
			\node at (5,0) {\begin{tikzpicture}
					\parkfunc{6}{1}{6,6,4,4,4,1}{4,5,1,2,6,3}{.7}
					\draw[blue,very thick] (.35,1.05) circle (.3);
					\draw[blue,very thick] (1.75,1.75) circle (.3);
					\draw[blue,very thick] (.6,.8) -- (.7,.7) -- (1.5,1.5);
			\end{tikzpicture}};
			\node at (10,0) {\begin{tikzpicture}
					\parkfunc{6}{1}{6,6,4,4,4,1}{4,5,1,2,6,3}{.7}
					\draw[blue,very thick] (.35,1.05) circle (.3);
					\draw[blue,very thick] (3.85,3.85) circle (.3);
					\draw[blue,very thick] (.6,.8) -- (.7,.7) -- (3.6,3.6);
			\end{tikzpicture}};
		\end{tikzpicture}
		\caption{All the diagonal inversions of a parking function $f \in \PF_6$}\label{fig:cl_exa_dinv}
	\end{figure}
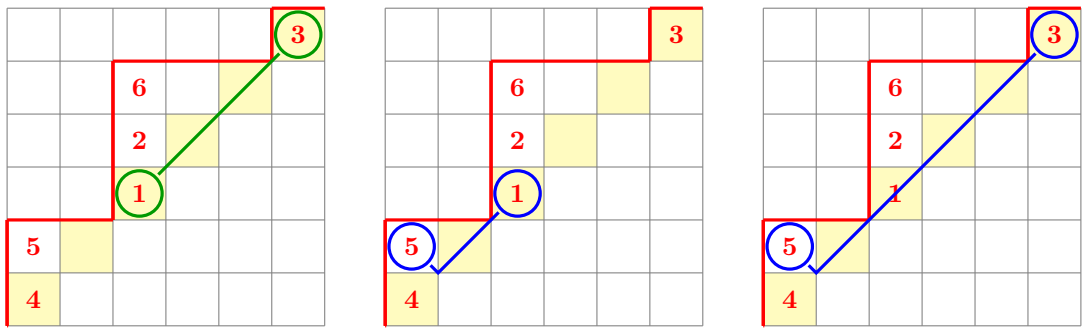
\end{example}

The pmaj statistic was introduced by Loehr and Remmel in~\cite{LoehrRemmel_pmaj} as an extension to labelled Dyck paths of the \emph{bounce} statistic for Dyck paths introduced by Haglund. 

\begin{definition}\label{def:cl_pmaj}
Given a parking function $f \in \PF_n$, define a permutation $\sigma_{\pmaj}(f) = \sigma_1\sigma_2\cdots\sigma_n \in \Sy_n$ in one-line notation as follows.\\
Set $B_0' \defeq \varnothing$ and $\sigma_0 := n+1$. Then iterate for $m = 1,2,\dots,n$ the following steps.
	\begin{enumerate}[label=\roman*)]
		\item Set $B_m \defeq B_{m-1}' \cup f^{-1}(m)$.
		\item Consider $X_m \defeq \{a \in B_m \ | \ a < \sigma_{m-1}\}$. There are two cases: if $X_m \neq \varnothing$ then define $\sigma_m \defeq \max(X_m)$, otherwise set $\sigma_m \defeq \max(B_m)$.
		\item Set $B_m' \defeq B_m \setminus \{\sigma_m\}$.
	\end{enumerate}
	It is easy to check that this algorithm is well defined, i.e.\ it always terminates: by definition of parking functions the set $B_m$ can never be empty.\\
	We define the \emph{pmaj} of $f$ as
	\[
	\pmaj(f) \defeq \maj(\sigma_{\pmaj}(f)^{\rev})
	\]
	where $\sigma_{\pmaj}(f)^{\rev} \defeq \sigma_n\sigma_{n-1}\cdots\sigma_1$ (the reverse word), and for any permutation $\tau\in \Sy_n$
	\[\maj(\tau):=\sum_{i\in [n-1]\: :\: \tau(i)>\tau(i+1)}i \]
	is its \emph{major index}. In particular, for any given label $\lambda \in [n]$ we can define its \emph{pmaj contribution} to $f$ as:
	\[
	p_\lambda(f) \defeq \#\{\text{ascents to the left of $\lambda$ in $\sigma_{\pmaj}(f)$}\},
	\]
	where $i\in [n-1]$ is an \emph{ascent} of $\tau\in \Sy_n$ if $\tau(i)<\tau(i+1)$.\\
	Hence, by definition of major index, we have $\pmaj(f) = \sum\limits_{\lambda \in [n]}p_\lambda(f)$.
\end{definition}

\begin{example}\label{exa:cl_pmaj_idea}
Considering $f$ as a labelled Dyck path, we can think of running this algorithm by reading at the $m$-th iteration loop the labels in column $m$ (again, we number the columns from $1$ to $n$, from left to right), i.e.\ the numbers in $f^{-1}(m)$. In Figure~\ref{fig:cl_exa_pmaj} we describe a step-by-step computation of the $\pmaj$ of a parking function of size $6$.
	\begin{figure}[ht]
		\centering
		\begin{tikzpicture}
			\node at (0,0) {\begin{tikzpicture}
					\parkfunc{6}{1}{6,6,5,5,2,2}{1,4,3,5,2,6}{.7}
			\end{tikzpicture}};
			\node[left] at (-2.2,-2.5) {$\sigma_{\pmaj}(f)$:};
			\node at (-1.75,-2.5) {4};
			\node at (-1.05,-2.5) {3};
			\node at (-.35,-2.5) {1};
			\draw[blue,very thick] (0,-2.3) -- (0,-2.8);
			\node at (.35,-2.5) {5};
			\draw[blue,very thick] (1.4,-2.3) -- (1.4,-2.8);
			\node at (1.05,-2.5) {2};
			\node at (1.75,-2.5) {6};
			\node[left] at (-2.2,-3.1) {$p_\lambda(f)$:};
			\node at (-1.75,-3.1) {0};
			\node at (-1.05,-3.1) {0};
			\node at (-.35,-3.1) {0};
			\node at (.35,-3.1) {1};
			\node at (1.05,-3.1) {1};
			\node at (1.75,-3.1) {2};
			\node at (6.5,0) {
				\renewcommand{\arraystretch}{1.4}
				\begin{tabular}{l|l l l l}
					$m$ & $B_m$ & $X_m$ & $\sigma_m$ & $B_m'$\\
					\hline
					1 & $\{1,4\}$ & $\{1,4\}$ & 4 & $\{1\}$\\
					2 & $\{1,3,5\}$ & $\{1,3\}$ & 3 & $\{1,5\}$\\
					3 & $\{1,5\}$ & $\{1\}$ & 1 & $\{5\}$\\
					4 & $\{5\}$ & $\varnothing$ & 5 & $\varnothing$\\
					5 & $\{6,2\}$ & $\{2\}$ & 2 & $\{6\}$\\
					6 & $\{6\}$ & $\varnothing$ & 6 & $\varnothing$
				\end{tabular}
			};
			\node at (3,-3.1) {$\Longrightarrow$};
			\node at (6.5,-3.1) {$\pmaj(f) = 0+0+0+1+1+2 = 4$};
		\end{tikzpicture}
		\caption{A step-by-step calculation of the $\pmaj$ statistic of an $f \in \PF_6$. On the left, the positions of the ascents of $\sigma_{\pmaj}(f)$ are highlighted with blue vertical bars.}\label{fig:cl_exa_pmaj}
	\end{figure}
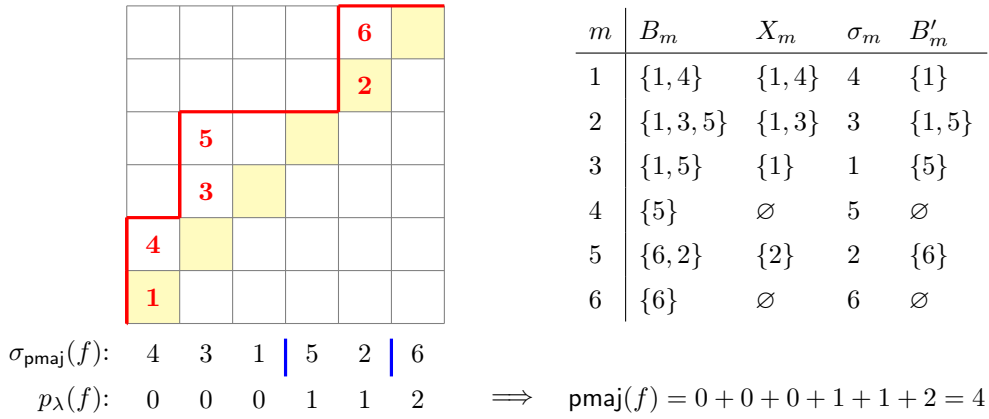
\end{example}

\section{The map $\phi_n$}\label{sec:cl_themap}

In~\cite{LoehrRemmel_pmaj} Loehr and Remmel defined a remarkable bijection of $\PF_n$ into itself, sending the bi-statistic $(\dinv,\area)$ to $(\area,\pmaj)$.

The function $\phi_n: \PF_n \longrightarrow \PF_n$ that we define in this section turns out to be the same function, though at first sight it looks quite different. In particular, our definition is more direct, hence easier to compute and to work with. This new definition is the only original element of the present work.

\medskip

In order to define our function $\phi_n:\PF_n\to \PF_n$, we need to introduce the \emph{dinv reading order} $\lread$ on the labels of $f \in \PF_n$: given $\lambda,\mu\in [n]$,
\[
\lambda \lread \mu \qquad \Longleftrightarrow \qquad a_\lambda(f) < a_\mu(f) \quad\text{or}\quad a_\lambda(f) = a_\mu(f) \text{ with } f(\lambda) < f(\mu).
\]
We define the \emph{dinv reading word} $\wdread(f) = w_1w_2\cdots w_n$ where the labels $\{1,2,\dots,n\}$ are ordered increasingly with respect to $\ldread$.\smallskip\\
Pictorially, the dinv reading word $\wdread(f)$ is simply obtained by reading the labels of $f$ along the diagonals $y=x+a$, starting from the main diagonal $y=x$ and going up, reading each diagonal from left to right. In Figure~\ref{fig:cl_exa_phi} the dinv reading word of the parking function on the left is computed.

\medskip

Let $\wdread(f)= w_1w_2 \cdots w_n$ be the dinv reading word of $f$. The \emph{dinv contribution} $d_\lambda(f)$ of the label $\lambda = w_j$ is
\begin{equation}\label{eq:cl_def_di}
	d_\lambda(f) \defeq \#\big\{ \{w_i,w_j\} \in \Dinv(f) \ \big| \ i < j \big\}.
\end{equation}
In Figure~\ref{fig:cl_exa_phi} the dinv contributions of the labels of the parking function on the left are computed.

\medskip

Finally, we define the image $\phi_n(f)$ of a parking function $f\in \PF_n$ with dinv reading word $\wdread(f)= w_1w_2 \cdots w_n$ as:
\begin{equation}\label{eq:cl_phi_def}
	(\phi_n(f))(w_i) \defeq (i-1) - d_{w_i}(f) + 1 = i - d_{w_i}(f) \qquad \text{for all $i \in [n]$}.
\end{equation}

\begin{example}
	Consider the parking function $f \in \PF_6$ from Examples~\ref{exa:cl_area}~and~\ref{exa:cl_dinv}, shown in Figure~\ref{fig:cl_exa_area}. Recall that its diagonal inversion set is
	\[
	\Dinv(f) = \{\{1,3\}, \{1,5\}, \{3,5\}\}.
	\]
	The dinv reading word $\wdread(f)$ is computed in Figure~\ref{fig:cl_exa_phi} by scanning the diagonals, as shown on the left. In the same figure, we wrote down the dinv contributions associated to the labels and used them to find the values of the function $\phi_6(f)$. In Figure~\ref{fig:cl_exa_phi} the resulting parking function $\phi_6(f)$ is shown on the right.
	\begin{figure}[ht]
		\centering
		\begin{tikzpicture}
			\node at (-2.3,0) {$f \defeq $};			
			\node at (12.5,0) {$=\phi_6(f)$};
			\node[scale=.85] at (0,0) {\begin{tikzpicture}
					\parkfunc{6}{1}{6,6,4,4,4,1}{4,5,1,2,6,3}{.7}
					\draw[blue,dotted,->,very thick] (.5,.5) -- (1.6,1.6);
					\draw[blue,dotted,->,very thick] (1.6,1.6) -- (3.7,3.7);
					\draw[blue,dotted,->,very thick] (3.7,3.7) -- (4.55,4.55);
					\draw[blue,dotted,left hook->,very thick] (-.35,.35) -- (.2,.9);
					\draw[blue,dotted,->,very thick] (.2,.9) -- (1.6,2.3);
					\draw[blue,dotted,->,very thick] (1.6,2.3) -- (3.85,4.55);
					\draw[blue,dotted,left hook->,very thick] (-.35,1.05) -- (1.6,3);
			\end{tikzpicture}};
			\node at (4.9,1) {$\Dinv(f) = \big\{\{1,3\},\{1,5\},\{3,5\}\big\}$};
			\node at (4.9,-1) {
				\renewcommand{\arraystretch}{1.4}
				\begin{tabular}{c|cccccc}
					{\color{blue}$\wdread(f)$} & {\color{blue}4} & {\color{blue}1} & {\color{blue}3} & {\color{blue}5} & {\color{blue}2} & {\color{blue}6}\\
					$d_{w_i}$ & 0 & 0 & 1 & 2 & 0 & 0\\
					{\color{red}$(\phi_6(f))(w_i)$} & {\color{red}1} & {\color{red}2} & {\color{red}2} & {\color{red}2} & {\color{red}5} & {\color{red}6}
				\end{tabular}
			};
			\node[scale=.85] at (10,0) {\begin{tikzpicture}
					\parkfunc{6}{1}{6,5,5,5,2,1}{4,1,3,5,2,6}{.7}
			\end{tikzpicture}};
		\end{tikzpicture}
		\caption{An example of the construction of the parking function $\phi_6(f)$.}\label{fig:cl_exa_phi}
	\end{figure}
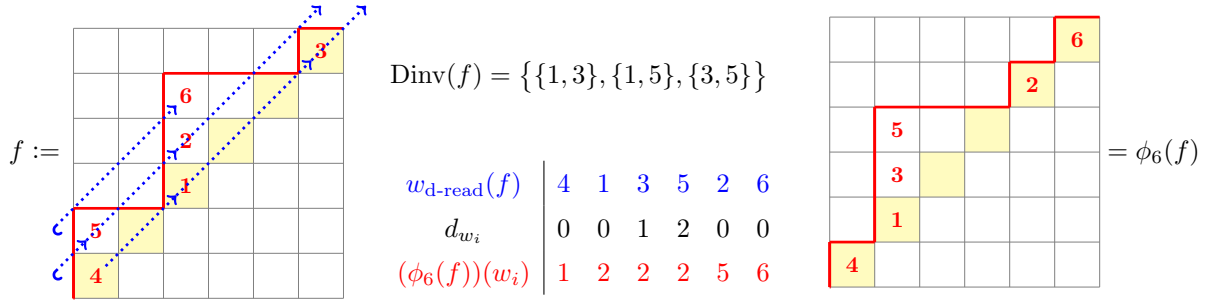
\end{example}

The goal of the rest of this article is to show that $\phi_n$ is a bijection sending the bistatistic $(\dinv,\area)$ to the bistatistic $(\area,\pmaj)$.

\smallskip

A priori, it is not even clear that $\phi_n(f)$ is indeed a parking function.

\begin{proposition}\label{prop:cl_phi_welldefined}
	The map $\phi_n:\PF_n\to \PF_n$ is well defined.
\end{proposition}

\begin{proof}
We need to show that $\phi_n(f)$ is a parking function.\\
		Since there are at most $i-1$ labels preceding $w_i$ in the total order $\ldread$, by definition
		\[
		0 \leq d_{w_i}(f) \leq i-1 \qquad \text{for all }i \in [n]
		\]
		and therefore by Equation~\eqref{eq:cl_phi_def} we have
		\[
		1 \leq (\phi_n(f))(w_i) \leq i \qquad \text{for all }i \in [n].
		\]
		This shows that $\phi_n(f)$ is a function from $[n]$ into itself. We now observe that for every $i\in [n]$
		\[
		\{w_1,\dots,w_i\} \subseteq \phi_n(f)^{-1}(\{1,\dots,i\}),
		\]
		therefore $|\phi_n(f)^{-1}(\{1,\dots,i\})| \geq i$. Hence $\phi_n(f)\in \PF_n$, as claimed.
\end{proof}

There is another property that is easy to deduce from our definition of $\phi_n$.
\begin{proposition}
For every $f \in \PF_n$ we have
\begin{equation}\label{eq:cl_phi_dinv-area}
	\dinv(f) = \area(\phi_n(f)).
\end{equation}
\end{proposition}
\begin{proof}
Given a parking function $f \in \PF_n$, we observed in Remark~\ref{rem:cl_pf_Dyck} that $f(\lambda)$ is the number of the column of $\lambda\in [n]$, hence, considering the diagram of $f$, 
\begin{equation}\label{eq:cl_intepr_phif}
	f(\lambda) = \#\left\{\begin{array}{c}
		\text{unit squares to the left of}\\
		\text{the north step labelled }\lambda
	\end{array}\right\} - 1 \qquad \text{for all }\lambda \in [n].
\end{equation}
The number of all the unit squares above the main diagonal in the diagram is $\binom{n}{2}$, therefore by definition of the area statistic
\begin{align*}
	\area(\phi_n(f)) &= \binom{n}{2} - \sum_{i=1}^{n}\big((\phi_n(f))(w_i)-1\big) = \sum_{i=1}^{n}(i-1) - \sum_{i=1}^{n}((i-1) - d_{w_i}(f))\\
	&= \sum_{i=1}^{n}d_{w_i}(f) = \#\Dinv(f) = \dinv(f).
\end{align*}
This concludes the proof.
\end{proof}

It remains to show that $\phi_n$ is bijective, and that $\area(f)=\pmaj(\phi_n(f))$ for every $f\in \PF_n$. These properties are trickier to prove. We address both of them ``simultaneously'' in the next sections.

\section{Definition of the inverse map $\psi_n$}\label{sec:cl_theinverse}

In this section we construct a map $\psi_n: \PF_n \to \PF_n$, whose definition is iterative, hence much more involved than the one of $\phi_n$ (cf.\ Remark~\ref{rem:speculation}). In later sections we will prove that $\psi_n$ is the inverse of $\phi_n$, and as a corollary we will deduce that $\area(f) = \pmaj(\phi_n(f))$ for all $f \in \PF_n$. 

\medskip

The definition of $\psi_n$ is essentially the original one appearing in~\cite{LoehrRemmel_pmaj}. Since it is quite involved, we will illustrate it via an example.

\begin{example}\label{exa:cl_constr_psi}
	Consider the parking function $g \in \PF_8$ in Figure~\ref{fig:cl_exa_psi_f}.
	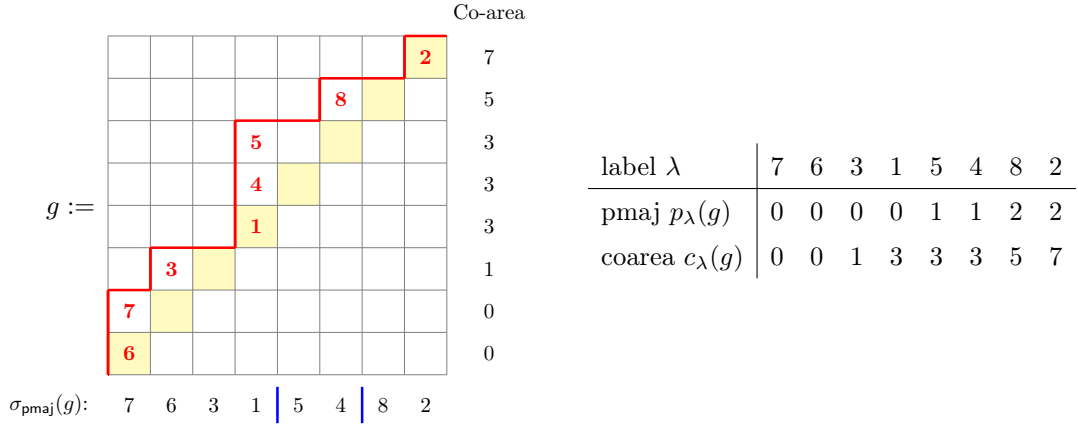
\begin{figure}[ht]
		\centering
		\begin{tikzpicture}
			\node at (-2.6,0) {$g \defeq $};
			\node[scale=.8] at (0,0) {
				\begin{tikzpicture}
					\parkfunc{8}{1}{8,8,7,5,5,5,3,1}{6,7,3,1,4,5,8,2}{.7}
					\node[left] at (-.2,-.5) {$\sigma_{\pmaj}(g)$:};
					\node at (0.35,-.5) {7};
					\node at (1.05,-.5) {6};
					\node at (1.75,-.5) {3};
					\node at (2.45,-.5) {1};
					\node at (3.15,-.5) {5};
					\node at (3.85,-.5) {4};
					\node at (4.55,-.5) {8};
					\node at (5.25,-.5) {2};
					\draw[blue,very thick] (2.8,-.2) -- (2.8,-.8);
					\draw[blue,very thick] (4.2,-.2) -- (4.2,-.8);
					\node at (6.3, 6) {Co-area};
					\node at (6.3, .35) {0};
					\node at (6.3, 1.05) {0};
					\node at (6.3, 1.75) {1};
					\node at (6.3, 2.45) {3};
					\node at (6.3, 3.15) {3};
					\node at (6.3, 3.85) {3};
					\node at (6.3, 4.55) {5};
					\node at (6.3, 5.25) {7};
				\end{tikzpicture}
			};
			\node[right] at (4,0) {
				\renewcommand{\arraystretch}{1.4}
				\begin{tabular}{l|c c c c c c c c}
					label $\lambda$ & 7 & 6 & 3 & 1 & 5 & 4 & 8 & 2 \\ \hline
					pmaj $p_\lambda(g)$ & 0 & 0 & 0 & 0 & 1 & 1 & 2 & 2 \\
					coarea $c_\lambda(g)$ & 0 & 0 & 1 & 3 & 3 & 3 & 5 & 7
				\end{tabular}
			};
		\end{tikzpicture}
		\caption{The parking function $g \in \PFp_8$ considered in the example and its $(\area,\pmaj)$ statistics. On the right a table with the information needed to construct $\psi_8(g)$.}\label{fig:cl_exa_psi_f}
	\end{figure}\\
	First of all, we compute the word $\sigma_{\pmaj}(g)$, and for each label $\lambda$ we record its $\pmaj$ contribution \mbox{$p_\lambda:=p_\lambda(g)$}, together with its \emph{coarea contribution} $c_\lambda=c_\lambda(g):=g(\lambda)-1$, i.e.\ the number of squares to the left of $\lambda$ in the diagram of $g$: cf.\ the table on the right in Figure~\ref{fig:cl_exa_psi_f}.

	The idea of the construction is to build up the labelled Dyck path $\psi_8(g)$ iteratively, by inserting the labels of the permutation $\sigma_{\pmaj}(g) = \sigma_1\sigma_2 \cdots \sigma_8$ in this order.\\
	We start with the empty $0\times 0$ labelled Dyck path. For every $m = 1,2,\dots,8$, we will choose the end point of a step in the $(m-1)\times (m-1)$ labelled Dyck path obtained at the previous iterative step, where we will insert a north step labelled $\sigma_m$ immediately followed by an east step, in such a way that
	\begin{itemize}
		\item we obtain a $m\times m$ labelled Dyck path,
		\item the row-area contribution of $\sigma_m$ in $\psi_8(g)$ is equal to $p_{\sigma_m}(g)$, and
		\item the label $\sigma_m$ ``avoids'' to create exactly $c_{\sigma_m}(g)$ diagonal inversions; in other words, in the new $m\times m$ labelled Dyck path the label $\sigma_m$ does not create a diagonal inversion with exactly $c_{\sigma_m}(g)$ many labels.
	\end{itemize}
	\medskip 
	In our example, this iterative construction consists of $8$ iterative steps, shown in Figure~\ref{fig:cl_exa_psi_constr}:
	\begin{itemize}
		\item \underline{Step 1}: just construct the $1\times1$ labelled Dyck path where the only north step is labelled $\sigma_1 = 7$.
		\item \underline{Step 2}: since label $\sigma_2 = 6$ has $\pmaj$ contribution $p_6 = 0$, the corresponding north step in the labelled path will have $0$ area contribution. Hence, we consider the main diagonal, highlighted with a dotted line in the picture. Intersecting this diagonal with the labelled path we get $2$ intersection points shown in blue. Both are possible spots in which to add the label $\sigma_2 = 6$, but they would ``avoid'' a different number of diagonal inversions (indicated with blue labels in the figure).\\
		Since $c_6 = 0$, at the blue point labelled $0$ we insert a north step labelled $\sigma_2 = 6$ followed by an east step. In this way we obtain the second labelled Dyck path in Figure~\ref{fig:cl_exa_psi_constr}. Observe that the labels $6$ and $7$ do create one dinv, so we ``avoided'' $c_6 = 0$ diagonal inversions.
		\item \underline{Steps $3$ and $4$}: since $p_3=p_4 = 0$, these two steps are analogous to Step 2: the insertion point for each new label is indicated by a bar on the correct number of ``avoided'' diagonal inversions ($c_3=1$ and $c_1=3$).
		\item \underline{Step 5}: now $\sigma_5 = 5$ has $\pmaj$ contribution $p_5 = 1$. So in $\psi_8(g)$ it should have $1$ row-area contribution, i.e.\ it should lie exactly one diagonal above the main one (cf.\ the fourth labelled path in Figure~\ref{fig:cl_exa_psi_constr}). Intersecting this diagonal with the path, we get 4 intersection points, shown with a dot ``\textbullet'' or with a ``$\times$''.\\
		Since at every step we want a labelled Dyck path, not every intersection point is suitable for our insertion. For example label $5$ cannot lie directly above label $6$ or label $7$. We indicate with an ``$\times$'' these ``bad'' intersection points.\\
		We observe that inserting label $5$ in the remaining positions would avoid $3$ and $4$ diagonal inversions.\\
		In this case $c_5 = 3$, so we do our insertion at the blue point labelled $3$.
		\item \underline{Steps $6$, $7$ and $8$}: we follow the previous rules, keeping in mind to use the diagonal \mbox{$y=x+p_{\sigma_m}$} if $\sigma_m$ has $\pmaj$ contribution $p_{\sigma_m}$, and to avoid the ``bad'' intersection points.
	\end{itemize}
	\begin{figure}[ht]
		\centering
		\begin{tikzpicture}
			\node[scale=.9] at (0,0) {\begin{tikzpicture}
					\draw[green!60!black, dotted, very thick] (-.2,-.2) -- (.9,.9);
					\parkfunc{1}{1}{1}{7}{.7};
					\node[blue,label={[label distance=-.2cm,text=blue]135:$\overline{0}$}] at (0,0) {\textbullet};
					\node[blue,label={[label distance=-.2cm,text=blue]135:1}] at (.7,.7) {\textbullet};
			\end{tikzpicture}};
			
			\node[scale=.9] at (1.8,0) {\begin{tikzpicture}
					\draw[green!60!black, dotted, very thick] (-.2,-.2) -- (1.6,1.6);
					\parkfunc{2}{1}{2,1}{6,7}{.7};
					\node[blue,label={[label distance=-.2cm,text=blue]135:0}] at (0,0) {\textbullet};
					\node[blue,label={[label distance=-.2cm,text=blue]135:$\overline{1}$}] at (.7,.7) {\textbullet};
					\node[blue,label={[label distance=-.2cm,text=blue]135:2}] at (1.4,1.4) {\textbullet};
			\end{tikzpicture}};
			
			\node[scale=.9] at (4.4,0) {\begin{tikzpicture}
					\draw[green!60!black, dotted, very thick] (-.2,-.2) -- (2.3,2.3);
					\parkfunc{3}{1}{3,2,1}{6,3,7}{.7};
					\node[blue,label={[label distance=-.2cm,text=blue]135:0}] at (0,0) {\textbullet};
					\node[blue,label={[label distance=-.2cm,text=blue]135:1}] at (.7,.7) {\textbullet};
					\node[blue,label={[label distance=-.2cm,text=blue]135:2}] at (1.4,1.4) {\textbullet};
					\node[blue,label={[label distance=-.2cm,text=blue]135:$\overline{3}$}] at (2.1,2.1) {\textbullet};
			\end{tikzpicture}};
			
			\node[scale=.9] at (7.7,0) {\begin{tikzpicture}
					\draw[green!60!black, dotted, very thick] (-.2,.5) -- (2.3,3);
					\parkfunc{4}{1}{4,3,2,1}{6,3,7,1}{.7};
					\node[blue,label={[label distance=-.2cm,text=blue]135: }] at (0,.7) {$\times$};
					\node[blue,label={[label distance=-.2cm,text=blue]135:$\overline{3}$}] at (.7,1.4) {\textbullet};
					\node[blue,label={[label distance=-.2cm,text=blue]135: }] at (1.4,2.1) {$\times$};
					\node[blue,label={[label distance=-.2cm,text=blue]135:4}] at (2.1,2.8) {\textbullet};
			\end{tikzpicture}};
			
			\node[scale=.9] at (11.7,0) {\begin{tikzpicture}
					\draw[green!60!black, dotted, very thick] (-.2,.5) -- (3,3.7);
					\parkfunc{5}{1}{5,4,4,2,1}{6,3,5,7,1}{.7};
					\node[blue,label={[label distance=-.2cm,text=blue]135: }] at (0,.7) {$\times$};
					\node[blue,label={[label distance=-.2cm,text=blue]135:$\overline{3}$}] at (.7,1.4) {\textbullet};
					\node[blue,label={[label distance=-.2cm,text=blue]135:4}] at (1.4,2.1) {\textbullet};
					\node[blue,label={[label distance=-.2cm,text=blue]135: }] at (2.1,2.8) {$\times$};
					\node[blue,label={[label distance=-.2cm,text=blue]135:5}] at (2.8,3.5) {\textbullet};
			\end{tikzpicture}};
			
			\node[scale=.9] at (3,-4.7) {\begin{tikzpicture}
					\draw[green!60!black, dotted, very thick] (-.2,1.2) -- (3,4.4);
					\parkfunc{6}{1}{6,5,5,4,2,1}{6,3,4,5,7,1}{.7};
					\node[blue,label={[label distance=-.2cm,text=blue]135:$\overline{5}$}] at (.7,2.1) {\textbullet};
					\node[blue,label={[label distance=-.2cm,text=blue]135:6}] at (1.4,2.8) {\textbullet};
			\end{tikzpicture}};
			
			\node[scale=.9] at (9,-4.7) {\begin{tikzpicture}
					\draw[green!60!black, dotted, very thick] (-.2,1.2) -- (3.7,5.1);
					\parkfunc{7}{1}{7,6,6,6,4,2,1}{6,3,4,8,5,7,1}{.7};
					\node[blue,label={[label distance=-.2cm,text=blue]135: }] at (.7,2.1) {$\times$};
					\node[blue,label={[label distance=-.2cm,text=blue]135:$\overline{7}$}] at (1.4,2.8) {\textbullet};
					\node[blue,label={[label distance=-.2cm,text=blue]135: }] at (2.1,3.5) {$\times$};
			\end{tikzpicture}};
			
			\node at (3,-10.4) {$\psi_8(g) = $};
			\node[scale=.9] at (6.3,-10.4) {\begin{tikzpicture}
					\parkfunc{8}{1}{8,7,7,7,6,4,2,1}{6,3,4,8,2,5,7,1}{.7};
			\end{tikzpicture}};
		\end{tikzpicture}
		\caption{The sequence of labelled Dyck paths to construct $\psi_8(g)$ for the $g$ in Figure~\ref{fig:cl_exa_psi_f}.}\label{fig:cl_exa_psi_constr}
	\end{figure}
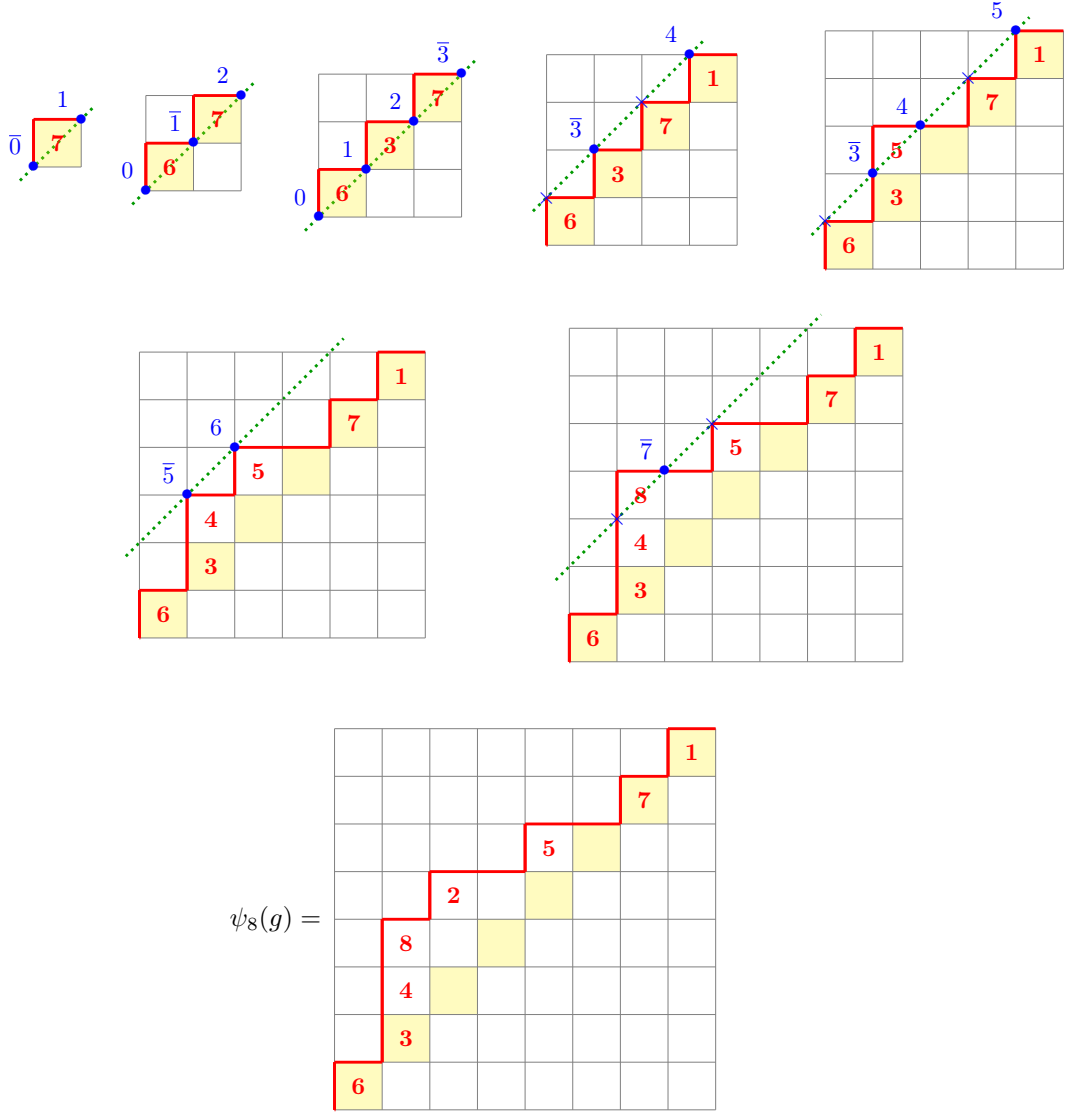
\end{example}

\newpage

\section{$\psi_n$ is well defined}

First of all we need to show that $\psi_n:\PF_n\to \PF_n$ is a well-defined function.

\smallskip

Consider $g \in \PF_n$ a parking function and compute $\sigma_{\pmaj}(g) = \sigma_1\sigma_2 \cdots\sigma_n$. Recall that for each label $\lambda \in [n]$, $c_\lambda=c_\lambda(g) \defeq g(\lambda) - 1$ denotes the co-area contribution of the row with north step labelled $\lambda$ in $g$, and $p_\lambda = p_\lambda(g)$ denotes the pmaj contribution of the label $\lambda$ in $g$.
\medskip\\
We want to perform our iterative procedure as in Example~\ref{exa:cl_constr_psi}, hence, starting with the empty labelled Dyck path, we will perform $n$ insertion steps. At Step $m$, we will look at the intersection of the diagonal $y=x+p_{\sigma_m}(g)$ with the $(m-1)\times (m-1)$ labelled Dyck path obtained at Step $m-1$, and insert a north step labelled $\sigma_m$ immediately followed by an east step in a \emph{suitable point} of that diagonal, i.e.\ an end point of a step of the $(m-1)\times (m-1)$ labelled Dyck path such that the insertion performed in that point gives a $m\times m$ labelled Dyck path. Notice that in this way its row-area contribution equals $p_{\sigma_m}(g)$. Moreover, we want to insert $\sigma_m$ in such a way that it avoids exactly $c_{\sigma_m}(g)$ diagonal inversions.\medskip

In order to show that $\psi_n$ is well defined we need the following lemma.

\begin{lemma}\label{lem:well_def}
For every $m=1,2,\dots,n$, at Step $m$ there exists a unique suitable insertion point on which the insertion of label $\sigma_m$ avoids exactly $c_{\sigma_m}$ diagonal inversions.
\end{lemma}

The rest of this section is dedicated to the proof of this fundamental result. We need a few definitions.\\
Given $g \in \PF_n$, consider the permutation $\sigma_{\pmaj}(g) = \sigma_1\sigma_2\cdots\sigma_n$, and insert a bar at the ascent positions of this word. The subwords separated by the bars are called \emph{runs}, and we denote them $R_0,R_1,R_2,\dots$ in increasing order from left to right. 
\begin{remark} \label{rem:Rpsigmam}
Observe that by construction $\lambda\in R_{p_\lambda(g)}$ for every $\lambda\in [n]$.	
\end{remark}

\begin{example}\label{exa:cl_desc_runs}
	Consider $\sigma_{\pmaj}(g)$ from Example~\ref{exa:cl_constr_psi}. We draw vertical bars to show the ascents of the word:
	\begin{center}
		\begin{tikzpicture}
			\node at (-1,0) {$\sigma_{\pmaj}(g)=$};
			\node at (0,0) {7};
			\node at (0.5,0) {6};
			\node at (1,0) {3};
			\node at (1.5,0) {1};
			\node at (2,0) {5};
			\node at (2.5,0) {4};
			\node at (3,0) {8};
			\node at (3.5,0) {2.};
			
			\draw[blue, very thick] (1.75,-.35) -- (1.75, .35);
			\draw[blue, very thick] (2.75,-.35) -- (2.75, .35);
		\end{tikzpicture}
	\end{center}
	So $R_0=7\, 6\, 3\, 1$, $R_1=5\, 4$ and $R_2=8\, 2$.
\end{example}

Now add the letter $\sigma_0 = 0$ at the beginning of the word $\sigma_{\pmaj}(g) = \sigma_1\sigma_2\dots\sigma_n$, so that we now have a new ascent and a new run: $R_{-1} = 0$.  For every $m \in [n]$, if $\sigma_m$ occurs in $R_h$, then we set
\begin{equation}\label{eq:cl_def_um}
	u_m \defeq \#\{\text{labels in $R_h$ greater than $\sigma_m$}\} + \#\{\text{labels in $R_{h-1}$ smaller than $\sigma_m$}\},
\end{equation}
and define the word $u(g):=u_1 u_2\cdots u_n$.

\begin{example}
	Consider $\sigma_{\pmaj}(g)=7\, 6\, 3\, 1\, 5\, 4\, 8\, 2$ from Example~\ref{exa:cl_constr_psi}. Then $u(g)=1\, 2\, 3\, 4\, 2\, 3\, 2\, 1.$
\end{example}
From every $m\in [n]$, the $u_m$ provides a range in which $g(\sigma_m)$ can occur. 

\begin{lemma}\label{lem:cl_co_area}
	Given $g \in \PF_n$, consider the permutation $\sigma_{\pmaj}(g) = \sigma_1\sigma_2\cdots\sigma_n$. For all $m \in [n]$ we have
	\begin{equation}\label{eq:cl_dis_tech}
		m - u_m \leq c_{\sigma_m}(g) \leq m - 1\, ,
	\end{equation}
	equivalently
	\begin{equation*}
		m - u_m + 1 \leq g(\sigma_m) \leq m.
	\end{equation*}
\end{lemma}
\begin{proof}
	We prove the two inequalities separately.\\
	To show that $m - u_m + 1 \leq g(\sigma_m)$ there are two sub-cases.
	\begin{itemize}
		\item if $\sigma_m$ occurs in $R_0$, then $u_m = m$, and the inequality becomes $1 \leq g(\sigma_m)$ which is always true;
		\item if $\sigma_m$ occurs in $R_h$ with $h > 0$, then suppose that $g(\sigma_m) \leq m - u_m$: in this case the algorithm in Definition~\ref{def:cl_pmaj} would have inserted the label $\sigma_m$ earlier in the construction of $\sigma_{\pmaj}(g)$, which is absurd.
	\end{itemize}
	To prove $g(\sigma_m) \leq m$, we look at the algorithm in Definition~\ref{def:cl_pmaj}: the inequality must hold since $\sigma_m \in B_m \subseteq g^{-1}(\{1,2,\dots,m\})$.\\
	This concludes the proof.
\end{proof}

We can now state a more precise lemma.
\begin{lemma}\label{lem:cl_avoid_dinv}
	Let $g \in \PF_n$ and consider the iterative procedure to construct $\psi_n(g)$.
	For every $m \in [n]$ there are exactly $u_m$ suitable insertion points on the labelled Dyck path of size $(m-1)\times(m-1)$, and the insertion procedure of $\sigma_m$ would avoid respectively $m - u_m$, $m - u_{m} + 1$, $\dots$, $m-1$ diagonal inversions, reading these points from left to right.
\end{lemma}
\begin{proof}
	The key observation is that the suitable points along the diagonal $y=x+p_{\sigma_m}$ are precisely the points where our insertion procedure would sit $\sigma_m$ on top of a smaller label on the diagonal $y=x+p_{\sigma_m}-1$, or right after\footnote{In the case $p_{\sigma_m} = 0$ it might also be at the beginning of the path: this will be counted by the run $R_{-1} = 0$.} a (necessarily bigger) label on the diagonal $y=x+p_{\sigma_m}$. Since by Remark~\ref{rem:Rpsigmam} we have $\sigma_m\in R_{p_{\sigma_m}}$, this gives precisely $u_m$ suitable points. Now the final statement on the number of diagonal inversions ``avoided'' by these insertions follows easily from their definitions.
\end{proof}

We can finally prove our first lemma.
\begin{proof}[Proof of Lemma~\ref{lem:well_def}]
	It follows immediately by combining Lemma~\ref{lem:cl_co_area} and Lemma~\ref{lem:cl_avoid_dinv}.
\end{proof}

\section{Proof of $\psi_n=\phi_n^{-1}$}

In order to prove that $\psi_n$ is the inverse of $\phi_n$, the following remark is crucial.
\begin{remark}\label{rem:dlambda}
	Given $g\in \PF_n$, in our iterative procedure defining $\psi_n(g)$, the insertions of $\sigma_{m'}$ for $m'>m$ do not affect the number of labels $\sigma_j$ with $j<m$ which $\sigma_m$ creates a diagonal inversion with. This is clear, since the insertion of $\sigma_{m'}$ does not change their relative positions, i.e.\ the conditions (A) and (B) in Definition~\ref{def:cl_dinv}.
\end{remark}

\begin{theorem}
	The function $\psi_n$ is the inverse of $\phi_n$.
\end{theorem}
\begin{proof}
	Since these are maps from the finite set $\PF_n$ into itself, it is enough to prove that $\phi_n\circ \psi_n$ is the identity function of $\PF_n$.	\\
	Consider a parking function $g \in \PF_n$. We want to show that $\big(\phi_n \circ \psi_n(g)\big)(\lambda) = g(\lambda)$ for any given label $\lambda \in [n]$. Set $f = \psi_n(g) \in \PF_n$.\\
	Consider the pmaj word $\sigma_{\pmaj}(g) = \sigma_1\sigma_2\cdots\sigma_n$ of $g$ and the reading word $\wdread(f) = w_1w_2\cdots w_n$ of $f$.\\
	By construction of $f = \psi_n(g)$, for every $m\in [n]$ we have
	\begin{equation}\label{eq:cl_pres_area}
		p_{\sigma_m}(g) = a_{\sigma_m}(f).
	\end{equation}
	Moreover, given $m \in [n]$, if we set
	\[
		X:=\{ \lambda \in \{\sigma_1,\sigma_2,\dots, \sigma_{m-1}\} \ | \ \lambda\text{ ``avoids'' dinv with }\sigma_m\text{ in }f \big\},
	\]
	then using Remark~\ref{rem:dlambda}, by construction of $\psi_n$ we have
	\begin{align}\label{eq:cl_coarea_phipsi}
		g(\sigma_m) - 1 = c_{\sigma_m}(g) = \# X.            
	\end{align}
	Suppose that $\sigma_m = w_k$, and let $\sigma_m \in R_h$.
	\begin{remark}\label{rem:Rhaf}
	Combining Remark~\ref{rem:Rpsigmam} and~\eqref{eq:cl_pres_area} we have $\sigma_m\in R_{h}$ if and only if $h=a_{\sigma_m}(f)$.
	\end{remark}
	\noindent We want to show that $X$ coincides with the set
	\[
		Y \defeq \big\{ \lambda \in \{w_1,\dots,w_{k-1}\} \ | \ \lambda\text{ ``avoids'' dinv with }w_k = \sigma_m\text{ in }f \big\}.
	\]
	There are three cases:
	\begin{itemize}
		\item $\lambda \in R_{h'}$ with $h' > h=a_{\sigma_m}(f)$. In this case, by Remark~\ref{rem:Rhaf}, the number $\lambda$ occurs to the right of $\sigma_m=w_k$ in both $\sigma_{\pmaj}(g)$ and $\wdread(f)$. Hence we have $\lambda \notin X$ and $\lambda \notin Y$, because $\lambda \notin \{\sigma_1,\dots,\sigma_{m-1}\}$ and $\lambda \notin \{w_1,\dots,w_{k-1}\}$.
		
		\item $\lambda \in R_{h'}$ with $h' < h=a_{\sigma_m}(f)$. In this case, by Remark~\ref{rem:Rhaf}, the number $\lambda$ occurs to the left of $\sigma_m=w_k$ in both $\sigma_{\pmaj}(g)$ and $\wdread(f)$. Hence clearly $\lambda\in X$ if and only if $\lambda\in Y$, since the ``avoidance'' condition is the same in both $X$ and $Y$.
		\item $\lambda \in R_{h}$ (recall $a_{\lambda}(f) = h$). We have two cases:
		\begin{itemize}
			\item[$\triangleright$] $\lambda > \sigma_m = w_k$: in this case $\lambda$ occurs to the left of $\sigma_m = w_k$ in $\sigma_{\pmaj}(g)$. So, since $\sigma_m<\lambda$ and $a_{\sigma_m}(f)=a_{\lambda}(f)=h$, we have that $\lambda\in X$ if and only if $f(\lambda)<f(\sigma_m)$. But since $a_{\sigma_m}(f)=a_{\lambda}(f)$, and $\lambda>\sigma_m$, we have that $\lambda \in Y$ if and only if $f(\lambda)<f(\sigma_m)$ (so that $\lambda$ occurs to the left of $\sigma_m=w_k$ in $\wdread(f)$).
			\item[$\triangleright$] $\lambda < \sigma_m = w_k$: in this case $\lambda$ occurs to the right of $\sigma_m$ in $\sigma_{\pmaj}(g)$, hence $\lambda\notin X$. On the other hand, if $\lambda$ occurs to the right $w_k=\sigma_m$ in $\wdread(f)$, then clearly $\lambda \notin Y$; while if it occurs to the left of $w_k=\sigma_m$ in $\wdread(f)$, then we must have $f(\lambda) < f(\sigma_m)$ since $a_{\sigma_m}(f)=a_{\lambda}(f)=h$. Together with $\lambda<\sigma_m=w_k$, this gives $\{\lambda, w_k\} \in \Dinv(f)$, hence $\lambda \notin Y$ as well.
		\end{itemize}
	\end{itemize}
	This shows that $X = Y$. Restarting from Equation~\eqref{eq:cl_coarea_phipsi}, using Equations~\eqref{eq:cl_def_di} and~\eqref{eq:cl_phi_def} from the construction of $\phi_n$, we obtain that for all $m \in [n]$:
	\[
		g(\sigma_m) - 1 = \#X = \#Y = (m-1) - d_{\sigma_m}(f) = \big(\phi_n(f)\big)(\sigma_m) - 1,
	\]
	which shows that $g = \phi_n(f) = \phi_n\circ\psi_n(g)$.\\
	This completes the proof of the theorem.
\end{proof}

\section{Further properties of $\phi_n$}

First of all, we deduced the following property that we mentioned earlier in this article.

\begin{theorem}
	For every $f\in \PF_n$ we have
	\begin{equation}
		\pmaj(\phi_n(f))=\area(f).
	\end{equation}
\end{theorem}
\begin{proof}
	Since $\psi_n$ is the inverse of $\phi_n$, we can prove the equivalent statement that for every $g\in \PF_n$
	\[
		\area(\psi_n(g))=\pmaj(g).
	\]
	But this is a straightforward consequence of Equation~\eqref{eq:cl_pres_area}.
\end{proof}
There is another property of $\phi_n$ that is worth mentioning here. The proof of the following lemma is straightforward and it is left to the reader.
\begin{lemma}
	Let $f\in \PF_n$ and $i\in [n-1]$. Let $\wdread(f)=w_1w_2\cdots w_n$, and let $\{i,i+1\}=\{w_k,w_{k+r}\}$  with $r>0$. Then
	\[
		\{\lambda\in\{w_1,\dots,w_{k-1}\}\mid \lambda\text{ ``avoids'' dinv with }w_k \}\subseteq \{\lambda\in\{w_1,\dots,w_{k-1}\}\mid \lambda\text{ ``avoids'' dinv with }w_{k+r} \}.
	\]
\end{lemma}
The following corollary is an immediate consequence of the previous lemma and the definition of $\phi_n$.
\begin{corollary}\label{cor:phininequalities}
	Let $f\in \PF_n$ and $i\in [n-1]$. Let $\wdread(f)=w_1w_2\cdots w_n$, and let $i=w_k$ and $i+1=w_{k+r}$. If $r>0$, then we have
	\[
		\phi_n(f)(i)\leq \phi_n(f)(i+1),
	\]
	while if $r<0$, then
	\[
		\phi_n(f)(i+1)\leq 1+\phi_n(f)(i).
	\]
\end{corollary}
The relevance of these results is in the following consequences.\\
We define the \emph{pmaj reading word} $\wpread(f)= w_1w_2\cdots w_n$ where $w_i$ is simply the the label of $f$ in its $i$-th row (counting as usual the rows from bottom to top, from $1$ to $n$).\\
For example the $f$ in Figure~\ref{fig:cl_exa_phi} (on the left) has $\wpread(f)= 4\, 5\, 1\, 2\, 6\, 3$.

\smallskip

The property of $\phi_n$ that we want to mention is the following.

\begin{proposition}\label{prop:phinshuffle}
	Let $f\in \PF_n$ and $i\in [n-1]$. Then the relative position of $i$ and $i+1$ in $\wdread(f)$ is the same as in $\wpread(\phi_n(f))$.
\end{proposition}
\begin{proof}
	If $i$ occurs to the left of $i+1$ in $\wdread(f)$, then Corollary~\ref{cor:phininequalities} says that $i$ occurs in $\phi_n(f)$ in a column weakly to the left of the column containing $i+1$. Hence clearly $i+1$ must occur in row that is higher than the row containing $i$, i.e.\ $i$ occurs to the left of $i+1$ in $\wpread(\phi_n(f))$.

	Similarly, if $i$ occurs to the right of $i+1$ in $\wdread(f)$, then Corollary~\ref{cor:phininequalities} says that $i$ occurs in $\phi_n(f)$ in a column strictly to the right of the column containing $i+1$. Hence clearly $i+1$ must occur in row that is lower than the row containing $i$, i.e.\ $i$ occurs to the right of $i+1$ in $\wpread(\phi_n(f))$.
\end{proof}
This property of $\phi_n$ has the following relevant consequence.

\medskip

Recall that a composition $\mu = (\mu_1,\mu_2,\dots,\mu_k)$ is simply a vector of positive integers, and we denote by $|\mu|:=\mu_1+\mu_2+\cdots+\mu_k$ its \emph{size}, and by $\ell(\nu):=k$ its \emph{length}.

Given two compositions $\mu = (\mu_1,\mu_2,\dots)$ and $\nu = (\nu_1,\nu_2,\dots)$ with $\lvert \mu \rvert + \lvert \nu \rvert=n$, let $K_{1} = \{n,n-1,\dots, n-\mu_1+1\}$, $K_{2} = \{n-\mu_1, n-\mu_1-1,\dots,n-\mu_1-\mu_2+1\}$, and so on, and let $I_{1} = \{1,2,\dots, \nu_1\}$, $I_{2} = \{\nu_1+1,\nu_1+2,\dots, \nu_1+\nu_2\}$, and so on. Notice that the sets $K_{1}, K_{2}, \dots$, $I_{1}, I_{2}, \dots$ form a partition of $[n]$.

Now let $\uparrow\!\! K_{i}$ be the word consisting of the elements of $K_{i}$ in increasing order: for example $\uparrow\!\! K_{1}=(n-\mu_1+1)\: (n-\mu_1+2) \cdots (n-1)\: n$.
Similarly, let $\downarrow\!\! I_{j}$ be the word consisting of the elements of $I_{j}$ in decreasing order: for example $\downarrow\!\! I_{1}=\nu_1 (\nu_1-1)\cdots 2\: 1$.

Consider the shuffle \[ \mathsf{W}(\mu;\nu) \coloneqq \uparrow\!\! K_{1} \shuffle \uparrow\!\! K_{2} \shuffle \cdots \shuffle \uparrow\!\! K_{\ell(\mu)} \shuffle \downarrow\!\! I_{1} \shuffle \downarrow\!\! I_{2} \shuffle \cdots \shuffle \downarrow\!\! I_{\ell(\nu)}, \]
i.e.\ the set of permutations in $\mathfrak{S}_n$ (in one-line notation) such that the letters in each $\uparrow\!\! K_{i}$ occur in the same relative order in $\uparrow\!\! K_{i}$, and similarly for each $\downarrow\!\! I_{j}$.

For example, if $\mu=(2,1)$ and $\nu=(1,2)$, then $|\mu|+|\nu|=3+3=6=n$, $K_1=\{5,6\}$, $K_2=\{4\}$, $I_1=\{1\}$, $I_2=\{2,3\}$, $\uparrow\!\! K_{1}=5\: 6$, $\uparrow\!\! K_{2}=4$, $\downarrow\!\! I_{1}=1$, $\downarrow\!\! I_{2}=3\: 2$, and $\mathsf{W}(\mu;\nu)=\mathsf{W}((2,1);(1,2))$ is the set of permutations in $\mathfrak{S}_6$ (in one-line notation), in which $5$ occurs to the left of $6$ and $3$ occurs to the left of $2$ (so that $\#\mathsf{W}((2,1);(1,2))=180$).

Given two compositions $\mu = (\mu_1,\mu_2,\dots)$ and $\nu = (\nu_1,\nu_2,\dots)$ with $\lvert \mu \rvert + \lvert \nu \rvert=n$, let $\PF(\mu;\nu)$ be the set of parking functions $f\in \PF_n$ such that $\wdread(f)\in \mathsf{W}(\mu;\nu)$, and let $\overline{\PF}(\mu;\nu)$ be the set of parking functions $f\in \PF_n$ such that $\wpread(f)\in \mathsf{W}(\mu;\nu)$.\\
Then Proposition~\ref{prop:phinshuffle} implies that $\phi_n$ restricts to a bijection from $\PF(\mu;\nu)$ to $\overline{\PF}(\mu;\nu)$, proving that
\[\sum_{f\in \PF(\mu;\nu)}q^{\dinv(f)}t^{\area(f)}=\sum_{f\in \overline{\PF}(\mu;\nu)}q^{\area(f)}t^{\pmaj(f)}.\]
It turns out that this polynomial equals $\langle\nabla e_n,e_{\mu}h_{\nu}\rangle$, and from this it is not hard to show that indeed the bijection $\phi_n$ explains combinatorially the equivalence of the original shuffle conjecture for $\nabla e_n$ (which was in terms of $(\dinv,\area)$), and its reformulation in terms of $(\area,\pmaj)$ due to Loehr and Remmel.

\medskip

We conclude this article with a speculative remark.
\begin{remark}\label{rem:speculation}
	It does not seem plausible to us that an easier (non recursive) definition of the inverse of $\phi_n$ is available. Indeed, let us consider the restriction of $\phi_n$ to the parking functions of area zero. These are essentially permutations, and here the dinv statistic reduces to the number of coinversions (i.e.\ the inversions of the reverse word, in one-line notation). In this case the map $\phi_n$ computes first what is essentially the Lehmer code of the reverse word (in one-line notation), and then reorders the entries so that the corresponding encoding labels appear in weakly increasing order. See Figure~\ref{fig:exampl_phi_area0} for an example.
	
	\begin{figure}[ht!]
		\centering
		\begin{tikzpicture}
			\node[scale=.85] at (0,0) {\begin{tikzpicture}
					\parkfunc{6}{1}{6,5,4,3,2,1}{4,5,1,2,6,3}{.7}
			\end{tikzpicture}};
			\node at (4.9,1.3) {$\rightarrow$
				\renewcommand{\arraystretch}{1.4}
				\begin{tabular}{c|cccccc}
					{\color{blue}$\wdread$} & {\color{blue}4} & {\color{blue}5} & {\color{blue}1} & {\color{blue}2} & {\color{blue}6} & {\color{blue}3}\\
					{\color{red}$(\phi_6(f))(w_i)$} & {\color{red}1} & {\color{red}1} & {\color{red}3} & {\color{red}3} & {\color{red}1} & {\color{red}4}
				\end{tabular}
			};
			\node at (6.2,0) {$\downarrow$ };
			\node at (4.9,-1.2) {\hspace{3.2cm}
				\renewcommand{\arraystretch}{1.4}
				\begin{tabular}{cccccc}
					{\color{blue}4} & {\color{blue}5} & {\color{blue}6} & {\color{blue}1} & {\color{blue}2} & {\color{blue}3}\\
					{\color{red}1} & {\color{red}1} & {\color{red}1} & {\color{red}3} & {\color{red}3} & {\color{red}4}
				\end{tabular}
				$\longleftrightarrow$};
			\node[scale=.85] at (10.8,0) {\begin{tikzpicture}
					\parkfunc{6}{1}{6,6,6,4,4,3}{4,5,6,1,2,3}{.7}
			\end{tikzpicture}};
		\end{tikzpicture}
		\caption{An example of the construction of $\phi_6(f)$ where $\area(f)=0$.}\label{fig:exampl_phi_area0}
	\end{figure}
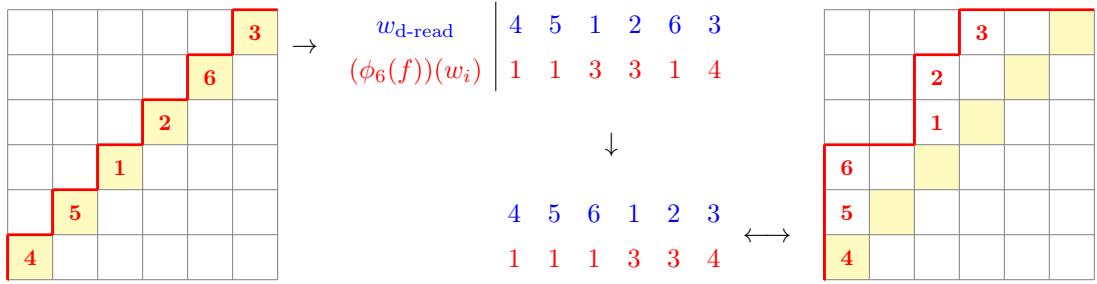
	
	\noindent Hence the computation of the inverse of $\phi_n$ requires a ``decoding'', which seems intrinsically more involved. See Figure~\ref{fig:exampl_psi_area0}.
	
	\begin{figure}[ht!]
		\centering

		\begin{tikzpicture}
			\node[scale=.85] at (0,0) {\begin{tikzpicture}
					\parkfunc{6}{1}{6,6,6,4,4,3}{4,5,6,1,2,3}{.7}
				\end{tikzpicture}
			};
		\end{tikzpicture}

		\begin{tikzpicture}
			\node at (0,0) { $\longleftrightarrow$
				\renewcommand{\arraystretch}{1.4}
				\begin{tabular}{cccccc}
					{\color{blue}4} & {\color{blue}5} & {\color{blue}6} & {\color{blue}1} & {\color{blue}2} & {\color{blue}3}\\
					{\color{red}1} & {\color{red}1} & {\color{red}1} & {\color{red}3} & {\color{red}3} & {\color{red}4}
				\end{tabular} $\cdots \rightarrow$ \begin{tabular}{cccccc}
					{\color{blue}4} & {\color{blue}5} & {\color{blue}1} & {\color{blue}2} & {\color{blue}6} & {\color{blue}3}\\
					{\color{red}1} & {\color{red}1} & {\color{red}3} & {\color{red}3} & {\color{red}1} & {\color{red}4}
				\end{tabular}
			};
		\end{tikzpicture}
		\caption{``Decoding'' $\phi_6(f)$ to recover $f$ from Figure~\ref{fig:exampl_phi_area0}.}\label{fig:exampl_psi_area0}
	\end{figure}
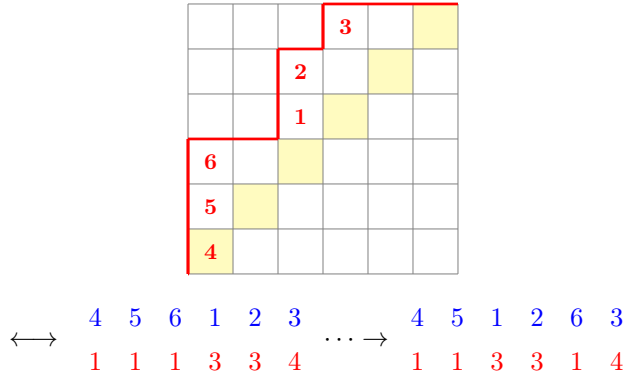
\end{remark}

\bibliographystyle{amsalpha}
\bibliography{bibliogr}

\end{document}